\pdfoutput=1
\documentclass[11pt,a4paper]{article}

\usepackage[T1]{fontenc}
\usepackage[utf8]{inputenc}
\usepackage{lmodern}
\usepackage[margin=1in]{geometry}
\usepackage{amsmath,amssymb,amsthm}
\usepackage{graphicx}
\usepackage{microtype}
\usepackage{authblk}
\usepackage{tikz}
\usetikzlibrary{arrows}
\usepackage[hidelinks]{hyperref}
\hypersetup{
  pdftitle={On a conjecture about the Sombor index of graphs},
  pdfauthor={Kinkar Chandra Das, Ali Ghalavand, Ali Reza Ashrafi},
  pdfkeywords={Sombor index, Reduced Sombor index, First Zagreb index, Extremal problem}
}

\numberwithin{equation}{section}
\newtheorem{thm}{Theorem}[section]
\newtheorem{lem}[thm]{Lemma}
\newtheorem{cor}[thm]{Corollary}
\newtheorem{con}{Conjecture}

\title{On a conjecture about the Sombor index of graphs}
\author[a]{Kinkar Chandra Das\thanks{Corresponding author: \href{mailto:kinkardas2003@gmail.com}{kinkardas2003@gmail.com}.}}
\author[b]{Ali Ghalavand}
\author[b]{Ali Reza Ashrafi}
\affil[a]{Department of Mathematics, Sungkyunkwan University, Suwon 16419, Republic of Korea;
\href{mailto:kinkardas2003@gmail.com}{\texttt{kinkardas2003@gmail.com}}}
\affil[b]{Department of Pure Mathematics, Faculty of Mathematical Sciences, University of Kashan, Kashan 87317--53153, I. R. Iran;
\href{mailto:alighalavand@grad.kashanu.ac.ir}{\texttt{alighalavand@grad.kashanu.ac.ir}},
\href{mailto:ashrafi@kashanu.ac.ir}{\texttt{ashrafi@kashanu.ac.ir}}}
\date{}

\begin{document}
\maketitle

\begin{abstract}
Let $G$ be a graph with vertex set $V(G)$ and edge set $E(G)$. { A graph invariant for $G$ is a number related to the structure of $G$ which is invariant under the symmetry of $G$}.
 The Sombor and reduced  Sombor indices of $G$  are { two new graph invariants} 
  defined as $SO(G)=\sum_{uv\in E(G)}\sqrt{d_G(u)^2 + d_G(v)^2}$ and $SO_{red}(G)=\sum_{uv\in E(G)}\sqrt{\big(d_G(u)-1\big)^2 + \big(d_G(v)-1\big)^2}$,  respectively, {where $d_G(v)$ is the degree of the vertex $v$ in $G$}. We denote by $H_{n,\nu}$ the graph constructed from the star $S_n$ by adding $\nu$ edge(s), $0\leq \nu\leq n-2$, between a fixed pendent vertex and $\nu$ other pendent vertices. R\'eti et al. [T. R\'eti, T Do\v{s}li\'c and A. Ali, On the Sombor index of graphs, \textit{Contrib. Math.} \textbf{3} (2021) 11--18] proposed a conjecture that the graph $H_{n,\nu}$ has the maximum Sombor index among all connected $\nu$-cyclic graphs of order $n$, where $0\leq\,\nu\leq n-2$. {{In some earlier works the validity of this conjecture was proved for $\nu\leq 5$. In this paper we confirm that this conjecture is true, when $\nu = 6$. The Sombor index in the case that the number of pendent vertices is less than or equal to $n - \nu - 2$ {is} investigated and the same results {are} obtained for the reduced  Sombor index. }} Some relationships between Sombor, reduced  Sombor and  first  Zagreb indices of graphs {are} also obtained.

\end{abstract}

\noindent\textbf{Keywords:} Sombor index, Reduced Sombor index, First Zagreb index, Extremal problem.

\smallskip
\noindent\textbf{Mathematics Subject Classification:} 05C07, 05C09, 05C35.

\section{Basic Definitions}

Throughout this paper, all graphs considered are finite, undirected and simple. Let $G$ be such a graph with vertex set $V = V(G)$ and edge set $E = E(G)$. We recall that the degree of a vertex $v$ in $G$, $d_G(v)$, is defined as the number of edges incident to $v$. {Let $\Delta$ be the maximum degree in $G$}. The set of all vertices adjacent to the vertex $v$ is denoted by $N(G,v)$.  The edge degree of $e \in E(G)$ is the degree of  $e$ in the line graph of $G$.  If we define  $\varepsilon_i = \varepsilon_i(G)$ to be the number of  edges of degree $i$ in $G$, then it can be easily seen that  $\sum_{i=0}^{2\Delta(G)-2} \varepsilon_i = |E(G)|$. We also use the notation $m_{i,j}(G)$ for the number of edges of $G$ with endpoints of degrees $i$ and $j$.

A graph $G$ with this property that the degree of each vertex is at most four is called a chemical graph.  Suppose $V(G)=\{v_1,v_2,\ldots,v_n\}$ and $d_G(v_1)\geq\,d_G(v_2)\geq\cdots\geq\,d_G(v_n)$. Then the sequence
$d(G)=(d_G(v_1),d_G(v_2),\ldots,d_G(v_n))$ is called the degree sequence of $G$. The graph union $G \cup H$ of two graphs $G$ and $H$ with disjoint vertex sets is another graph with $V(G \cup H) = V(G) \cup V(H)$  and $E(G \cup H) = E(G) \cup E(H)$. The union of $s$ disjoint copies is denoted by $s\,G$. {{The join $G + H$ of two graphs $G$ and $H$ with disjoint vertex sets is a graph with vertex set $V(G \cup H)$ and edge set $E(G \cup H)$ $\cup$ $\{ uv \mid u \in V(G) \ {\wedge}  \ v \in V(H) \}$.}}

{Suppose $f$ is an edge  and $u, v$ are two non-adjacent vertices of $G$. Then $G-f$ is the subgraph of $G$ obtained by deleting the edge $f$ and $G + uv$ is a graph obtained from $G$ by adding an edge connecting $u$ and $v$.}

{{Suppose $G$ is a connected graph with exactly $n$ vertices and $m$ edges. The cyclomatic number of $G$ is defined as $\nu(G) = m - n + 1$ and a $\nu$-cyclic graph is a graph with cyclomatic number $\nu$}}.  For terms and notations not defined here we follow the standard texts in graph theory as the famous book of West \cite{18}.

The  first  Zagreb index of a graph $G$ is an old degree-based graph invariant introduced by Gutman and Trinajsti\'c \cite{7} defined as $M_{1}(G)$ $=$ $\sum_{ uv \in E(G)}[d_G(u)+ d_G(v)]$ $=$ $\sum_{v\in V(G)}d_G(v)^2$. In a recent paper about the general form of all degree-based topological indices of graphs \cite{8}, Gutman  introduced two new invariants and invited researchers to investigate their mathematical properties and chemical meanings. He used the names  ``Sombor index" and ``reduced  Sombor index" for his new graphical invariants. The Sombor and  reduced  Sombor indices are   defined as follows:
\begin{eqnarray*}
SO(G)  &=&  \sum_{uv\in E(G)}\sqrt{d_G(u)^2 + d_G(v)^2},\\
SO_{red}(G)   &=&  \sum_{uv\in E(G)}\sqrt{\big(d_G(u)-1\big)^2 + \big(d_G(v)-1\big)^2}.
\end{eqnarray*}

Very recently Red\v{z}epovi\'c \cite{15} discussed chemical applicability on Sombor index of graphs. Specifically, the Sombor index was used to model entropy and enthalpy of vaporization of alkanes with satisfactory prediction potential, indicating that this topological index may be used successfully on modeling thermodynamic properties of compounds. In the same paper the level of the predicting power of the Sombor index and its degeneracy were tested.

We refer to \cite{14,19}, for more information on degree-based topological indices of graphs and their extremal problems.

Let $\Gamma(s,n)$ denote  the set of all {{decreasing}} real sequences $c=(c_1,c_2,\ldots,c_n)$ such that $\sum_{i=1}^nc_i = s$. Define a relation $ \preceq $ on $\Gamma(s,n)$ as follows: For two {{decreasing}} real sequences $c=(c_1,c_2,\ldots,c_n)$ and $d=(d_1,d_2,\ldots,d_n) $ in $\Gamma(s,n)$, we write $c\preceq\,d$ if and only if for each integer $k$, $1\leq\,k\leq\,n-1$, we have $\sum_{i=1}^kc_i\leq\sum_{i=1}^kd_i$. It is easy to see that  $(\Gamma(s,n),\preceq)$ is a partially ordered set. The partial order $ \preceq $ is called the majorization and if $c\preceq\,d$ then we say that $c$ {{is}} majorized by $d$. We refer the interested readers to consult the survey article  \cite{11} and the book \cite{12}, for more information on majorization theory and its applications in graph theory.

Suppose $ X \subseteq \mathbb{R}^n$ and $a, b \in X$ are different points in $X$. The line segment $\overline{ab}$ is the set of all points $\lambda a + (1-\lambda)b$, where $0 < \lambda < 1$. The set $X$ is said to be convex, if for every point $a, b \in X$, $\overline{ab} \subseteq X$. Let $ X \subseteq \mathbb{R}^n$ be convex. The function $f:X \longrightarrow \mathbb{R}$ is called a convex function, if for any $a, b \in X$ and $0 < \lambda < 1$, we have $f(\lambda a + (1-\lambda)b)$ $\leq$ $\lambda$ $f(a)$ $+$ $(1-\lambda) f(b)$. If $f$ is convex and we have strict inequality for all $a \ne b$, then we say the function is strictly convex. It is well-known that if $I$ is an open interval and $g : I \longrightarrow\, \mathbb{R}$ is a real twice-differentiable function on $I$, then $g$ is convex if and only if  for each $x\in\,I$,  $g^{\prime\prime}(x)\geq0$. The function  $g$ is strictly convex on $I$, if $g^{\prime\prime}(x)>0$ for all $x\in\,I$.

\section{Background Materials}

In \cite{8}, Gutman proved that among all $n$-vertex graphs, the empty graph $\overline{K_n}$ and the complete graph $K_n$ have the minimum and maximum Sombor indices, respectively. He also proved that if we restrict our attention to the $n$-vertex connected graph then the $n$-vertex path $P_n$ will attain the minimum Sombor index. He also proved in \cite{9} that  $M_1(G) \geq SO(G) \geq \frac{1}{\sqrt{2}}M_1(G)$, and if $G$ has $m$ edges then $M_1(G) - 2m \geq SO_{red}(G) \geq \frac{1}{\sqrt{2}}(M_1(G) - 2m)$.

Cruz et al. \cite{1} characterized the extremal graphs with respect to the Sombor index over the set of all {$n$-vertex} chemical graphs, connected chemical graphs, chemical trees, and hexagonal systems. Cruz and Rada \cite{2}, studied the extremal values of Sombor index over the set of all unicyclic and also bicyclic graphs of a given order. In a recent work, Das et al. \cite{3} obtained lower and upper bounds for the Sombor index of graphs based on some other graph parameters. Moreover, they obtained some relationships between Sombor index and the first and second Zagreb indices of graphs.

Deng et al. \cite{4}, investigated  the chemical importance of the Sombor index and obtained the extremal values of the reduced Sombor index for chemical trees.  Milovanovi\'c et al. \cite{13} investigated the relationship between Sombor index and Albertson index which is an old irregularity measure for graphs, and in  \cite{15}, Red$\breve{\rm z}$epovi\'c examined the predictive and discriminative potentials of Sombor and reduced Sombor indices of chemical graphs. Wang et al. \cite{17} investigated the relationships between the Sombor index and some degree based invariants, and obtained some Nordhaus-Gaddum type results. In \cite{DA0,DA1}, the authors presented some bounds on Sombor index of some class of graphs in terms of graph parameters and characterized the extremal graphs.  {{Ghanbari and Alikhani \cite{0} computed this index for certain graphs and  examined its effects under some operations on vertex and edge of $G$. As a consequence, they presented bounds for the Sombor index of join and corona product of two graphs.}}

The following lemma \cite{10} is useful  in some of our results.

\begin{lem}\label{maj1}
 Suppose $c=(c_1,c_2,\ldots,c_n)$ and $d=(d_1,d_2,\ldots,d_n)$ are two {{decreasing}} sequences of
real numbers. If $c \preceq\,d$, then for any convex function
$f$, $\sum_{i=1}^nf(c_i)\leq
\sum_{i=1}^nf(d_i).$ Furthermore, if $ c\prec\,d$ and
$f$ is a strictly convex function, then
$\sum_{i=1}^nf(c_i)< \sum_{i=1}^nf(d_i).$
\end{lem}

{{The graph constructed from the star $S_n$ by adding $\nu$ edge(s), $0\leq \nu\leq n-2$, between a fixed pendent vertex and $\nu$ other pendent vertices is denoted by $H_{n,\nu}$}} \cite{16}.

\begin{con} {\rm (R\'eti et al. \cite{16})} \label{con} If $\nu$ and $n$ are fixed integers satisfying the inequality $5\leq\,\nu\leq n- 2$ then among all connected $\nu$-cyclic
graphs of order $n$, only the graph $H_{n,\nu}$ has the maximum Sombor index.
\end{con}

\section{New Bounds on Sombor Index}

{{The aim of this section is to present some new bounds on Sombor and reduced Sombor indices of graphs.}}  The relationship between the Sombor index, reduced  Sombor index and the first  Zagreb index of graphs will also be investigated.

\begin{lem}\label{2lm1}
Let $G$ be an $n$-vertex connected graph with cyclomatic number $\nu$ and degree sequence
$d(G)=(d_G(v_1),d_G(v_2),\ldots,d_G(v_n))$ {with this property that the number of pendent vertices  is less than or equal to $n - \nu - 2$}. If $0\leq\nu\leq\,n-2$ and $d_G(v_1)=n-1$, then
\begin{equation}\label{eee1}
(d_G(v_2),d_G(v_3),\ldots,d_G(v_n))\preceq\, (\nu + 1, \overbrace{2, \ldots , 2}^\nu,\overbrace{1,\ldots,1}^{n-\nu-2}).
\end{equation}
\end{lem}

\begin{proof} {We have $\nu = m - n + 1$. Moreover,
   $$\sum\limits^n_{i=2}\,d_G(v_i)=2m - d_G(v_1)=2\nu + 2n-2-d_G(v_1)=2\nu + n - 1.$$
On the other hand, the sum of all integers in the right hand side of (\ref{eee1}) is again $2\nu + n - 1$. Let $c = (c_2, \ldots, c_n) = (d_G(v_2), \ldots, d_G(v_n))$ and let $d = (d_2, \ldots, d_n) = (\nu + 1, 2, \ldots, 2, 1, \ldots, 1)$, where the multiplicities of
 the numbers $1$ and $2$ in the last sequence are $n-\nu-2$ and $\nu$, respectively. Thus we have
 \begin{equation}
 \sum\limits^n_{i=2}\,c_i=\sum\limits^n_{i=2}\,d_G(v_i)=2\nu + n - 1=\sum\limits^n_{i=2}\,d_i.\label{p1kin1}
 \end{equation}
 Since the number of pendent vertices  is less than or equal to $n - \nu - 2$ in $G$, one can easily see that $c_i \geq d_i$ for $i=3,4,\ldots,n$. From this result with (\ref{p1kin1}), we conclude that $\sum_{i=2}^kc_i\leq\sum_{i=2}^kd_i$ $(2\leq\,k\leq\,n-1)$.
 Hence $c \preceq d$. This completes the proof of the result.}
\end{proof}

\begin{lem} {\rm (R\'eti et al. \cite{16})} \label{2lm2}
Suppose $G$ is a connected graph with maximum Sombor index among all connected graphs with $n$ vertices and cyclomatic number $\nu$. If
$0\leq\nu\leq\,n-2$, then $\Delta(G)=n-1$.
\end{lem}

{{Define $SO^{\ddagger}(G)=\sum_{uv\in\,E(G)}\sqrt{(d_G(u)+1)^2+(d_G(v)+1)^2}.$}}

\begin{lem}\label{th2}
Let $G$ be a graph with $n$ vertices and $m$ edges. Then
$$SO^{\ddagger}(G)\leq\,m\sqrt{(m+1)^2+4}.$$
 For $m\leq n-1$, the equality holds if and only if $G\cong\,S_{m+1}\cup\,(n-m-1)K_1$.
\end{lem}

\begin{proof} For any $uv\in\,E(G)$, we have $d_G(u)+d_G(v)\leq\,m+1$. Note that the function $g(x)=(x+1)^2$ is strictly convex on $(-\infty,\infty)$. For any edge $uv\in E(G)$, we have
\begin{equation}
\sqrt{(d_G(u)+1)^2+(d_G(v)+1)^2}\leq\sqrt{(d_G(u)+1+m+1-d_G(u)-d_G(v))^2+(d_G(v)+1)^2}\label{1q1}
\end{equation}
and
 $(d_G(u)+m+1-d_G(u)-d_G(v),\,d_G(v))\preceq\,(m,1)$. Moreover, the equality holds in (\ref{1q1}) if and only if $d_G(u)+d_G(v)=m+1$. For any $uv\in\,E(G)$, by Lemma \ref{maj1}, we have
 \begin{align*}
 \sqrt{(d_G(u)+1+m+1-d_G(u)-d_G(v))^2+(d_G(v)+1)^2}\leq\sqrt{(m+1)^2+4}
 \end{align*}
 with equality holding if and only if $(d_G(u)+m+1-d_G(u)-d_G(v),\,d_G(v))=(m,1)$. For any $uv\in\,E(G)$, we obtain
 \begin{align*}
 \sqrt{(d_G(u)+1)^2+(d_G(v)+1)^2}\leq\sqrt{(m+1)^2+4}
 \end{align*}
 with equality holding if and only if $(d_G(u),\,d_G(v))=(m,1)$. Hence
\begin{align*}
SO^{\ddagger}(G)&=\sum_{uv\in\,E(G)}\sqrt{(d_G(u)+1)^2+(d_G(v)+1)^2}\\[2mm]
&\leq\sum_{uv\in\,E(G)}\sqrt{(m+1)^2+4}\\[2mm]
&=m\sqrt{(m+1)^2+4}.
\end{align*}
Moreover, the above equality holds if and only if $G\cong\,S_{m+1}\cup\,(n-m-1)K_1$ as $m\leq n-1$.
\end{proof}

{
\begin{lem}\label{1kin1} {Let $n\geq 3$}. The two functions $f(x)=\sqrt{(n-1)^2+x^2}$ and {$g(x)=\sqrt{(n-2)^2+(x-1)^2}$} are strictly convex on $(-\infty,\infty)$.
\end{lem}}

{
\begin{proof} We have $f^{\prime\prime}(x)={\frac {(n-1)^2}{ \left( {x}^{2}+(n-1)^2\right) ^{\frac{3}{2}}}}$. Therefore for each $r$, $r\in(-\infty,\infty)$, we have $f^{\prime\prime}(r)>0$. This proves that $f$ is strictly convex on $(-\infty,\infty)$. Similarly,
one can easily prove that $g(x)=\sqrt{(n-2)^2+(x-1)^2}$ is strictly convex on $(-\infty,\infty)$.
\end{proof}}

\begin{lem}\label{2lm3}
Let $G$ be a connected graph with cyclomatic number $\nu$ $(0\leq\nu\leq\,n-2)$ and vertex set $V(G)=\{v_1,v_2,\ldots,v_n\}$ {with this condition that the number of pendent vertices is less than or equal to $n - \nu - 2$}. If
$d(G)=(n-1,d_G(v_2),\ldots,d_G(v_n))$, then
\begin{eqnarray*}
\sum\limits_{v\in\,V(G)\backslash\{v_1\}}\sqrt{(n-1)^2+d_G(v)^2}&\leq \left( n-\nu-2 \right) \sqrt { \left( n-1 \right) ^{2}+1} +\nu\sqrt { \left( n-1 \right) ^{2}+4}\\
&~~~~~~~~~~~~~~~~+\sqrt{ \left( n-1 \right) ^{2}+ \left(\nu+1 \right) ^{2}}
\end{eqnarray*}
with equality if and only if $G\cong\,H_{n,\nu}$.
\end{lem}

\begin{proof}
Suppose $f(x)=\sqrt{(n-1)^2+x^2}$. By Lemma \ref{1kin1}, $f$ is strictly convex on $(-\infty,\infty)$. By Lemma \ref{2lm1}, $(d_G(v_2),d_G(v_3),\ldots,d_G(v_n))\preceq\, (\nu + 1, {2, \ldots , 2},{1,\ldots,1}),$ where the multiplicities of the numbers $1$ and $2$ in the last sequence are $n - \nu - 2$ and $\nu$, respectively. Now Lemma \ref{maj1} implies that
$\sum_{v\in\,V(G)\backslash\{v_1\}}\sqrt{(n-1)^2+d_G(v)^2}$ $\leq$  $\left( n-\nu-2 \right) \sqrt { \left( n-1 \right) ^{2}+1}$ $+$ $\nu\sqrt { \left( n-1 \right) ^{2}+4}  +\sqrt { \left( n-1 \right) ^{2}+ \left(\nu+1 \right) ^{2}}$
with equality if and only if $G\cong\,H_{n,\nu}$.
\end{proof}

{{We are now ready to prove one of the main  {results} of this section.}}

\begin{thm}\label{thmain1}
Let $G$ be a connected graph with maximum value of Sombor index among all $n$-vertex connected graphs with cyclomatic number $\nu$ {and assume that the number of pendent vertices  is less than or equal to $n - \nu - 2$}. If
$0\leq\nu\leq\,n-2$ then $G\cong\,H_{n,\nu}$ and $SO(G)$ $=$ $\left( n-\nu-2 \right) \sqrt { \left( n-1 \right) ^{2}+1}+\nu\sqrt { \left( n-1 \right) ^{2}+4} + \sqrt{ \left( n-1 \right) ^{2}+ \left(\nu+1 \right) ^{2}}+\nu\sqrt { \left( \nu+1 \right) ^{2}+4}.$
\end{thm}

\begin{proof} By Lemma \ref{2lm2}, we have $\Delta(G)=n-1$. Suppose $u\in\,V(G)$ and $d_G(u)=n-1$. By definition of Sombor index, $SO(G)$ $=$ $\sum_{v\in\,V(G)\backslash\{u\}}\sqrt{(n-1)^2+d_G(v)^2}$ $+$ $SO^{\ddagger}(G-u)$. Since $G-u$ is a graph of order $n-1$ with $\nu$ edges, Lemmas \ref{th2} and \ref{2lm3} imply that
$SO(G)$ $\leq$ $\left( n-\nu-2 \right) \sqrt { \left( n-1 \right) ^{2}+1}$ $+$ $\nu\sqrt {
 \left( n-1 \right) ^{2}+4}$ $+$ $\sqrt { \left( n-1 \right) ^{2}+ \left(\nu+1 \right) ^{2}}+\nu\sqrt { \left( \nu+1 \right) ^{2}+4}$, with equality if and only if $G\cong\,H_{n,v}$. This completes the proof of the theorem.
\end{proof}

\begin{lem}\label{th1}
If $G$ is an $n$-vertex graph with exactly $m$ edges, then
$$SO(G)\leq\,m\sqrt{m^2+1}~\mbox{ and }~SO_{red}(G)\leq\,m(m-1)$$
with equalities  if and only if $m\leq\,n-1$ and $G\cong\,S_{m+1}\cup\,(n-m-1)K_1$.
\end{lem}

\begin{proof}
{{Suppose $uv\in\,E(G)$. Since $G$ is a graph with $m$ edges, $d_G(u)+d_G(v)\leq\,m+1$}}. On the other hand, the functions
$f(x)=x^2$ and $g(x)=(x-1)^2$ are strictly convex on $(-\infty,\infty)$, and increasing {on} $[1,\infty)$. Therefore,
\begin{align}
&\sqrt{d_G(u)^2+d_G(v)^2}\leq\sqrt{(d_G(u)+\lambda)^2+d_G(v)^2}\label{8.e5}\\
\mbox{ and }&~~~~~~~~~\nonumber\\
&\sqrt{(d_G(u)-1)^2+(d_G(v)-1)^2}\leq\sqrt{(d_G(u)-1+\lambda)^2+(d_G(v)-1)^2},\label{8.e6}
\end{align}
where $\lambda=m+1-(d_G(u)+d_G(v))$. The  equalities in \eqref{8.e5} and \eqref{8.e6} hold  if and only if $d_G(u)+d_G(v)=m+1$.
Furthermore, $(d_G(u)+\lambda,\,d_G(v))\preceq\,(m,1)$, and by Lemma \ref{maj1}, we obtain
\begin{align}
&\sqrt{(d_G(u)+\lambda)^2+d_G(v)^2}\leq\sqrt{m^2+1}\label{8.e7}\\
\mbox{ and }~~~~~&\nonumber\\
&\sqrt{(d_G(u)-1+\lambda)^2+(d_G(v)-1)^2}\leq (m-1).\label{8.e8}
 \end{align}
The  equalities in \eqref{8.e7} and \eqref{8.e8} hold  if and only if $(d_G(u)+\lambda,\,d_G(v))=(m,1)$. Therefore by  \eqref{8.e5}, \eqref{8.e6},  \eqref{8.e7} and \eqref{8.e8},
 \begingroup\allowdisplaybreaks\begin{align*}
SO(G)&=\sum_{uv\in\,E(G)}\sqrt{d_G(u)^2+d_G(v)^2}\leq\sum_{uv\in\,E(G)}\sqrt{m^2+1}=m\sqrt{m^2+1}\\[3mm]
\mbox{ and }~~~~~~~~~&\nonumber\\
SO_{red}(G)&=\sum_{uv\in\,E(G)}\sqrt{(d_G(u)-1)^2+(d_G(v)-1)^2}\leq\sum_{uv\in\,E(G)}\,(m-1)=m(m-1).
\end{align*}\endgroup
The equalities hold  if and only if $m\leq\,n-1$ and $G\cong\,S_{m+1}\cup\,(n-m-1)K_1$, proving the lemma.
\end{proof}

\begin{cor}\label{cor}
If $T$ is an $n$-vertex tree, then
$$SO(T)\leq\,(n-1)\sqrt{(n-1)^2+1}~\text{ and }~SO_{red}(T)\leq\,(n-1)(n-2).$$
The equalities hold if and only if $T\cong\,S_{n}$.
\end{cor}

\begin{lem}\label{8.l4}
Let $G$ be a graph with maximum reduced Sombor index among all connected graphs with $n$ vertices and cyclomatic number $\nu$.  Then $\Delta(G)=n-1$.
\end{lem}

\begin{proof}
Suppose $v\in\,V(G)$, $d_G(v)=\Delta(G)$, $N(G,v)=\{x_1,x_2,\ldots,x_\Delta\}$ and $d_G(v)\leq\,n-2$. Since $G$ is a  connected graph,   there exists a positive integer $i$, $1\leq\,i\leq\,\Delta$, such that {$A =N(G,x_i)\backslash\,(N(G,v)\cup \{ v\})\neq\emptyset$}. Assume that $G^\prime=G-\{x_ia|\,a\in\,A\}+\{va|\,a\in\,A\}$, $B=N(G,x_i)\cap\,N(G,v)$,
{$C =N(G,v)\backslash(B\cup \{ x_i\})$} and $\mu=SO_{red}(G)-SO_{red}(G^\prime)$. By definitions,
{{\begingroup\allowdisplaybreaks\begin{align}
\mu=&\sum_{a\in\,A}\sqrt{(d_G(a)-1)^2+(d_G(x_i)-1)^2}
+\sum_{b\in\,B}\sqrt{(d_G(b)-1)^2+(d_G(x_i)-1)^2}\nonumber\\[2mm]
&+\sum_{b\in\,B}\sqrt{(d_G(b)-1)^2+(\Delta-1)^2}
+\sum_{c\in\,C}\sqrt{(d_G(c)-1)^2+(\Delta-1)^2}\nonumber\\[2mm]
&+\sqrt{(d_G(x_i)-1)^2+(\Delta-1)^2}
-\Big[ \sum_{a\in\,A}\sqrt{(d_G(a)-1)^2+(\Delta+|A|-1)^2}\nonumber\\[2mm]
&+\sum_{b\in\,B}\sqrt{(d_G(b)-1)^2+(d_G(x_i)-|A|-1)^2}
+\sum_{b\in\,B}\sqrt{(d_G(b)-1)^2+(\Delta+|A|-1)^2}\nonumber\\[2mm]
&+\sum_{c\in\,C}\sqrt{(d_G(c)-1)^2+(\Delta+|A|-1)^2} +\sqrt{(d_G(x_i)-|A|-1)^2+(\Delta+|A|-1)^2}\Big]\nonumber\\[2mm]
&<\sum_{b\in\,B}\sqrt{(d_G(b)-1)^2+(d_G(x_i)-1)^2}
+\sum_{b\in\,B}\sqrt{(d_G(b)-1)^2+(\Delta-1)^2}\nonumber\\[2mm]
&+\sqrt{(d_G(x_i)-1)^2+(\Delta-1)^2}{-\Big[}
\sum_{b\in\,B}\sqrt{(d_G(b)-1)^2+(d_G(x_i)-|A|-1)^2}\nonumber\\[2mm]
&+\sum_{b\in\,B}\sqrt{(d_G(b)-1)^2+(\Delta+|A|-1)^2} +\sqrt{(d_G(x_i)-|A|-1)^2+(\Delta+|A|-1)^2}\Big]\label{8.e10}
\end{align}\endgroup}}

{On the other hand, the  functions $h_1(x)=\sqrt{(d_G(b)-1)^2+(x-1)^2}$,  $b\in B$, { and $h_2(x)=(x-1)^2$}   are  strictly convex on $(-\infty,\infty)$,  and $\big(\Delta,d_G(x_i)\big)\prec\big(\Delta+|A|,d_G(x_i)-|A|\big)$. Apply Lemma \ref{maj1} to obtain
{\begin{align}
&\sum_{b\in\,B}\sqrt{(d_G(b)-1)^2+(d_G(x_i)-1)^2}
+\sum_{b\in\,B}\sqrt{(d_G(b)-1)^2+(\Delta-1)^2}\nonumber\\[2mm]
&+\sqrt{(d_G(x_i)-1)^2+(\Delta-1)^2}<
\sum_{b\in\,B}\sqrt{(d_G(b)-1)^2+(d_G(x_i)-|A|-1)^2}\nonumber\\[2mm]
&+\sum_{b\in\,B}\sqrt{(d_G(b)-1)^2+(\Delta+|A|-1)^2} +\sqrt{(d_G(x_i)-|A|-1)^2+(\Delta+|A|-1)^2}.\label{8.e11}
\end{align}}}

Now by \eqref{8.e10} and \eqref{8.e11},  $SO_{red}(G)-SO_{red}(G^\prime)<0$, contradicts by the fact that $SO_{red}(G)$ is maximum. So, $\Delta(G)=n-1$. This completes the proof.
\end{proof}

\begin{lem}\label{8.lm5}
Let $G$ be a connected graph {of order $n\,(\geq 3)$} with cyclomatic number $\nu$ $(0\leq\nu\leq\,n-2)$ and vertex set $V(G)=\{v_1,v_2,\ldots,v_n\}$ {with this condition that the number of pendent vertices is less than or equal to $n - \nu - 2$}. If
$d(G)=(n-1,d_G(v_2),\ldots,d_G(v_n))$, then
{ \begingroup\allowdisplaybreaks\begin{align*}
\sum_{v\in\,V(G)\backslash\{v_1\}}\sqrt{(n-2)^2+(d_G(v)-1)^2}\leq&\left( n-\nu-2 \right)\left( n-2 \right)+\nu\sqrt { \left( n-2 \right) ^{2}+1}\\
&~~~~~~~~~~~~~~~~  +\sqrt { \left( n-2 \right) ^{2}+ \nu ^{2}}
\end{align*}\endgroup}
with equality if and only if $G\cong\,H_{n,\nu}$.
\end{lem}

\begin{proof}
{Suppose $g(x)=\sqrt{(n-2)^2+(x-1)^2}$. By Lemma \ref{1kin1},  $g$ is strictly convex on $(-\infty,\infty)$.} By Lemma \ref{2lm1}, $(d_G(v_2),d_G(v_3),\ldots,d_G(v_n))\preceq\, (\nu + 1, {2, \ldots , 2},{1,\ldots,1}),$ where the multiplicities of the numbers $1$ and $2$ in the last sequence are $n - \nu - 2$ and $\nu$, respectively. By Lemma \ref{maj1}, we obtain
 \begingroup\allowdisplaybreaks\begin{align*}
\sum_{v\in\,V(G)\backslash\{v_1\}}\sqrt{(n-2)^2+(d_G(v)-1)^2}\leq&\left( n-\nu-2 \right)\left( n-2 \right)+\nu\sqrt { \left( n-2 \right) ^{2}+1}\\
&  +\sqrt { \left( n-2 \right) ^{2}+ \nu ^{2}}
\end{align*}\endgroup
with equality if and only if $G\cong\,H_{n,\nu}$.
\end{proof}

{{The following theorem is the second main result of this section.}}

\begin{thm}\label{ath1}
{{Suppose $G$ has maximum reduced Sombor index among all $n$-vertex graphs with cyclomatic number $\nu$ and assume  that the number of pendent vertices is less than or equal to $n - \nu - 2$.}} If
$0\leq\nu\leq\,n-2$, then $G\cong\,H_{n,\nu}$ and $SO_{red}(G)=(n - {\nu} - 2)(n - 2)+{\nu} \sqrt { \left( n-2  \right) ^{2}+1}+{\nu}\sqrt {{{\nu}}^{2}+1}$ $+$ $\sqrt { \left( n-2 \right) ^{2}+{{\nu}}^{2}}$.
\end{thm}

\begin{proof}
 By Lemma \ref{8.l4}, we have $\Delta(G)=n-1$. Suppose $u\in\,V(G)$ and $d_G(u)=n-1$. By definition of reduced Sombor index, $SO_{red}(G)$ $=$ $\sum_{v\in\,V(G)\backslash\{u\}}\sqrt{(n-2)^2+(d_G(v)-1)^2}$ $+$ $SO(G-u)$. Since $G-u$ is a graph of order $n-1$ with $\nu$ edges, Lemmas \ref{th1} and \ref{8.lm5} imply that
{${SO_{red}(G)}$ $\leq$ $(n - {\nu} - 2)(n - 2)+{\nu} \sqrt { \left( n-2  \right) ^{2}+1}+{\nu}\sqrt {{{\nu}}^{2}+1}$ $+$ $\sqrt { \left( n-2 \right) ^{2}+{{\nu}}^{2}}$} with equality if and only if $G\cong\,H_{n,v}$.
\end{proof}

The following lemma is useful in finding some new lower bounds for the Sombor and reduced Sombor indices of graphs.

\begin{lem} {\rm (\cite{6})} \label{lm1}
{{If $G$ is a  graph with $n$ vertices, $m$ edges, and without isolated edges. Then}}
$\varepsilon_{1}(G)$ $=$ $4m-M_1(G)$ $+$ $\sum_{i=3}^{2n-4}\varepsilon_i(G)(i-2)$ and
$\varepsilon_{2}(G)$ $=$ $M_1(G)$ $-$ $3m$ $-$ $\sum_{i=3}^{2n-4}\varepsilon_i(G)(i-1)$.
\end{lem}

Suppose  $x$ and $y$ are two positive real numbers and  $x+y-2=s$. Since ${x^2+y^2}-\frac{(s+2)^2}{2}=x^2+y^2-\frac{(x+y)^2}{2}=\frac{1}{2}(x-y)^2$, ${x^2+y^2}\geq\frac{(s+2)^2}{2}$, and the equality holds if and only if $x=y$. So, $\sqrt{x^2+y^2}\geq\,\frac{\sqrt{2}(s+2)}{2}$, and the equality holds if and only if $x=y$.
{We now give two lower bounds on Sombor index and reduced Sombor index for connected graph $G$ in terms of $m$ and the first Zagreb index.}
\begin{thm}
{{Let $G$ be a connected graph with $n \geq 3$ vertices and $m$ edges. Then}}
\begin{align*}
SO(G)&\geq \frac{1}{3} \left(2\,\sqrt {2}- \sqrt {5} \right) \left( 3\,{\it M_1(G)}-4\,m+2\,\sqrt {5}\sqrt {2}m \right),\\
SO_{red}(G)&\geq \left( \sqrt {2}-1 \right) \left( {\it M_1(G)}-2\,m+\sqrt {2}m \right).
\end{align*}
The equalities hold if and only if $G\cong P_n$ or $G\cong C_n$.
\end{thm}

\begin{proof}
By definition of Sombor index,
\begingroup\allowdisplaybreaks\begin{align*}
SO(G)&=\sum_{1\leq\,i\leq\,j\leq\,n-1}m_{i,j}(G)\sqrt{i^2+j^2}\\
&=m_{1,2}\sqrt{5}+m_{1,3}\sqrt{10}+m_{2,2}\sqrt{8}+\sum_{j=4}^{n-1}m_{1,j}\sqrt{1+j^2}+\sum_{j=3}^{n-1}m_{2,j}\sqrt{4+j^2}\\
&+\sum_{3\leq i\leq j\leq n-1}m_{i,j}\sqrt{i^2+j^2}
\end{align*}\endgroup
{and by our discussion before the statement of this theorem,
\begingroup\allowdisplaybreaks\begin{align*}
SO(G)&\geq\,m_{1,2}\sqrt{5}+m_{1,3}\sqrt{10}+m_{2,2}\sqrt{8}+\sum_{j=4}^{n-1}m_{1,j}\frac{\sqrt{2}(j+1)}{2}
+\sum_{j=3}^{n-1}m_{2,j}\frac{\sqrt{2}(j+2)}{2}\\
&+\sum_{3\leq i\leq j\leq n-1}m_{i,j}\frac{\sqrt{2}(i+j)}{2}.
\end{align*}\endgroup
On the other hand,   $\sum_{1\leq i\leq j\leq n-1}m_{i,j}-(m_{1,2}+m_{1,3}+m_{2,2})=\sum_{i=3}^{2n-4}\varepsilon_i(G)$. Therefore, by
definition of $\varepsilon_i(G)$, $1\leq\,i\leq\,2n-4 $,}
\begingroup\allowdisplaybreaks\begin{align*}
SO(G)&\geq m_{1,2}\sqrt{5}+m_{1,3}\sqrt{10}+m_{2,2}\sqrt{8}+\sum_{i=3}^{2n-4}\varepsilon_i(G)\frac{\sqrt{2}(i+2)}{2}\\
&\geq m_{1,2}\sqrt{5}+ \big(m_{1,3}+m_{2,2}\big)\sqrt{8}+\sum_{i=3}^{2n-4}\varepsilon_i(G)\frac{\sqrt{2}(i+2)}{2}\\
&=\varepsilon_1(G)\sqrt{5}+\varepsilon_2(G)\sqrt{8}+\sum_{i=3}^{2n-4}\varepsilon_i(G)\frac{\sqrt{2}(i+2)}{2}.
\end{align*}\endgroup
{{The equality holds if and only if $m = m_{1,2} + \sum_{i=2}^{n-1}m_{i,i}$}}. Now, by Lemma \ref{lm1},
{\small\begingroup\allowdisplaybreaks\begin{align*}
SO(G)&\geq \Big[4m-M_1(G) + \sum_{i=3}^{2n-4}\varepsilon_i(G)(i-2)  \Big]\sqrt{5} +\sum_{i=3}^{2n-4}\varepsilon_i(G)\frac{\sqrt{2}(i+2)}{2}\\
&+\Big[M_1(G) -3m-\sum_{i=3}^{2n-4}\varepsilon_i(G)(i-1) \Big]\sqrt{8}\\
&=\Big[4m-M_1(G) \Big]\sqrt{5}+\Big[M_1(G) -3m \Big]\sqrt{8}
+\frac{1}{2}\sum_{i=3}^{2n-4}\varepsilon_i(G)\big(2\sqrt{5}-3\sqrt{2} \Big)\big(i-2 \big)\\
&\geq\frac{1}{3} \left(2\,\sqrt {2}- \sqrt {5} \right)  \left( 3\,{\it M_1(G)}-4\,m+2\,\sqrt {5}\sqrt {2}m \right).
\end{align*}\endgroup}
{{The equality holds if and only if $m = m_{1,2} + m_{2,2}$, that is, if and only if}} $G\cong P_n$ or $G\cong C_n$. By definition of reduced  Sombor index,
\begin{align*}
SO_{red}(G)&=\sum_{1\leq\,i\leq\,j\leq\,n-1}m_{i,j}(G)\sqrt{(i-1)^2+(j-1)^2}\\
&=m_{1,2}+2m_{1,3}+\sqrt{2}m_{2,2}+\sum_{j=4}^{n-1}m_{1,j}\sqrt{(j-1)^2}+\sum_{j=3}^{n-1}m_{2,j}\sqrt{1+(j-1)^2}\\
&~~~~~~~~~~~~~~~~~~~~~~~~~~~~~~~~~+\sum_{3\leq i\leq j\leq n-1}m_{i,j}\sqrt{(i-1)^2+(j-1)^2}.
\end{align*}
{Using the same technique in the proof of first part, we obtain,}
\begin{align*}
SO_{red}(G)&\geq m_{1,2}+2m_{1,3}+\sqrt{2}m_{2,2}+\sum_{i=3}^{2n-4}\varepsilon_i(G)\frac{\sqrt{2}i}{2}\\
&\geq m_{1,2}+\sqrt{2}\big(m_{1,3}+m_{2,2}\big)+\sum_{i=3}^{2n-4}\varepsilon_i(G)\frac{\sqrt{2}i}{2}\\
&=\varepsilon_1(G)+\sqrt{2}\varepsilon_2(G)+\sum_{i=3}^{2n-4}\varepsilon_i(G)\frac{\sqrt{2}i}{2}.
\end{align*}
{{The equality holds if and only if $m = m_{1,2} + \sum_{i=2}^{n-1}m_{i,i}$}}. Now by Lemma \ref{lm1},
{\small \begin{align*}
SO_{red}(G)&\geq \Big[4m-M_1(G) + \sum_{i=3}^{2n-4}\varepsilon_i(G)(i-2)  \Big] +\sum_{i=3}^{2n-4}\varepsilon_i(G)\frac{\sqrt{2}i}{2}\\
&~~~~~~~~~~~~~~~~~~~~~~~~~~+\sqrt{2}\Big[M_1(G) -3m-\sum_{i=3}^{2n-4}\varepsilon_i(G)(i-1) \Big]\\
&\geq\Big[4m-M_1(G) \Big]+\sqrt{2}\Big[M_1(G) -3m \Big]
+\frac{1}{2}\sum_{i=3}^{2n-4}\varepsilon_i(G)\big(2-\sqrt{2} \Big)\big(i-2 \big)\\
&\geq\,\left( \sqrt {2}-1 \right)  \left( {\it M_1(G)}-2\,m+\sqrt {2}m \right).
\end{align*}}
 {{The equality holds if and only if $m = m_{1,2} + m_{2,2}$, that is, if and only if}} $G\cong P_n$ or $G\cong C_n$.
\end{proof}

\section{The Case of \texorpdfstring{$\nu = 6$}{nu = 6}}

The aim of this section is to prove the case of $\nu = 6$ in Conjecture \ref{con}. { Also, we give some information on the structure of the graphs with
maximum Sombor index (resp. reduced Sombor index) among connected $n$-vertex $\nu$-cyclic graphs.To do this,
let $\mathbb{ES}$ be the set of graphs $G$ of order at least $1$ satisfying : for any vertices $v$ and $v^\prime$ of $G$, if $d_G(v)\leq d_G(v^\prime)$ then $N(G,v)$ is a subset of $N(G,v^\prime)\cup\{v^\prime\}$.} In addition, assume that $\mathbb{G}(n,\nu)$ denotes the set of all graphs with exactly $n$ vertices and cyclomatic number
$\nu$, where $0\leq\,\nu\leq\,n-2$. If $G\in\,\mathbb{G}(n,\nu)$, then we use that notation $n_1(G)$ for the number of pendent vertices of $G$.  If $0 \leq \nu\leq n-2$, then we define:
\begin{align*}
{\Omega_1^{n,\nu}}&=\{G\in\mathbb{G}(n,\nu)|\Delta(G)\leq\,n-2\},\\
{\Omega_2^{n,\nu}}&=\{G\in\mathbb{G}(n,\nu)|\Delta(G)=n-1~and~n_1(G)\leq\,n-\nu-2\},\\
{\Omega_3^{n,\nu}}&=\{G\in\mathbb{G}(n,\nu)|\Delta(G)=n-1,~n_1(G)\geq\,n-\nu-1~and~{G\not\in\mathbb{ES}}\},\\
{\Omega_4^{n,\nu}}&=\{G\in\mathbb{G}(n,\nu)|\Delta(G)=n-1,~n_1(G)\geq\,n-\nu-1~and~{G\in\mathbb{ES}}\},
\end{align*}
It is easy to see that {$\{\Omega_1^{n,\nu}, \Omega_2^{n,\nu}, \Omega_3^{n,\nu}, \Omega_4^{n,\nu}\}$} is a partition for $\mathbb{G}(n,\nu)$.

{ \begin{lem}\label{esl1}
Let $G$ be a graph with maximum  Sombor index among  connected graphs with $n$ vertices and cyclomatic number $\nu$.  Then $\Delta(G)=n-1$.
\end{lem}
The proof of Lemma \ref{esl1} is similar to that of Lemma \ref{8.l4}, so we omit it.}

{\begin{lem}\label{nlm2}
Let $G$ be a graph with maximum Sombor index (resp. reduced Sombor index) among
connected $n$-vertex $\nu$-cyclic graphs. Then $G$ is in $\mathbb{ES}$.
\end{lem}
\begin{proof}
We prove this result by contradiction. For this we assume that $u,v\in\,V(G)$, $d_G(u)\leq\,d_G(v)$ and
 $N(G,u)\backslash\,(N(G,v)\cup\{v\})=\{u_1,\ldots,u_r\}$, $r\geq1$. Assume that $G^\prime=G-\{uu_i|i=1,\ldots,r\}+\{vu_i|i=1,\ldots,r\}$,
 $A=N(G,u)\cap\,N(G,v)$, $\Lambda=SO(G)-SO(G^\prime)$
  and $\Upsilon=SO_{red}(G)-SO_{red}(G^\prime)$. By Lemmas \ref{8.l4} and \ref{esl1}, $G^\prime$ is a connected $n$-vertex $\nu$-cyclic graph.
   Now, via majorization technique one can  easily check  that $\sqrt{d_G(u)^2+d_G(v)^2}<\sqrt{(d_G(u)-r)^2+(d_G(v)+r)^2}$ and $\sqrt{(d_G(u)-1)^2+(d_G(v)-1)^2}<\sqrt{(d_G(u)-r-1)^2+(d_G(v)+r-1)^2}$.
  Therefore  by definitions,
{\small \begingroup\allowdisplaybreaks \begin{align}
  \Lambda&\leq \sum_{i=1}^r\sqrt{d_G(u)^2+d_G(u_i)^2}+\sum_{a\in\,A}\sqrt{d_G(u)^2+d_G(a)^2}
  +\sum_{a\in\,A}\sqrt{d_G(v)^2+d_G(a)^2}\nonumber\\
 & -\Big[ \sum_{i=1}^r\sqrt{(d_G(v)+r)^2+d_G(u_i)^2}+\sum_{a\in\,A}\sqrt{(d_G(u)-r)^2+d_G(a)^2}\nonumber\\
 &+\sum_{a\in\,A}\sqrt{(d_G(v)+r)^2+d_G(a)^2}\Big],\label{eseq1}\\[2mm]
  \Upsilon\leq& \sum_{i=1}^r\sqrt{(d_G(u)-1)^2+(d_G(u_i)-1)^2}+\sum_{a\in\,A}\sqrt{(d_G(u)-1)^2+(d_G(a)-1)^2}\nonumber\\
  &+\sum_{a\in\,A}\sqrt{(d_G(v)-1)^2+(d_G(a)-1)^2}
  -\Big[ \sum_{i=1}^r\sqrt{(d_G(v)+r-1)^2+(d_G(u_i)-1)^2}\nonumber\\
  &+\sum_{a\in\,A}\sqrt{(d_G(u)-r-1)^2+(d_G(a)-1)^2}+\sum_{a\in\,A}\sqrt{(d_G(v)+r-1)^2+(d_G(a)-1)^2}\Big].\label{eseq2}
  \end{align}\endgroup}
  The equalities hold if and only if $uv\not\in\,E(G)$. On the other hand,
\begin{align*}
\sum_{i=1}^r\sqrt{d_G(u)^2+d_G(u_i)^2}
&<\sum_{i=1}^r\sqrt{(d_G(v)+r)^2+d_G(u_i)^2},\\
\sum_{i=1}^r\sqrt{(d_G(u)-1)^2+(d_G(u_i)-1)^2}
&<\sum_{i=1}^r\sqrt{(d_G(v)+r-1)^2+(d_G(u_i)-1)^2},
\end{align*}
the functions $h_1(x)=\sqrt{x^2+d_G(a)^2}$ and $h_2(x)=\sqrt{(x-1)^2+(d_G(a)-1)^2}$, $a\in A$, are  strictly convex on $(-\infty,\infty)$,  and $\big(d_G(v),d_G(u)\big)\prec\big(d_G(v)+r,d_G(u)-r\big)$. By Lemma \ref{maj1}, from \eqref{eseq1} and \eqref{eseq2}, we obtain $SO(G)-SO(G^\prime)<0$ and
  $SO_{red}(G)-SO_{red}(G^\prime)<0$, contradicts by the fact that $SO(G)$ (resp. $SO_{red}(G)$)
  is maximum. So, $G$ is in $\mathbb{ES}$.  This completes the proof.
\end{proof}}

{ \begin{lem}\label{esl3}
 Let $G$ be a graph in $\mathbb{ES}$ of order at least $2$. Then $G$ has an isolated or universal vertex $u$,
with $G-\{u\}$ being in $\mathbb{ES}$.
 \end{lem}
 \begin{proof}
 Let $G$ be in $\mathbb{ES}$. Assume that $G$ has no isolated vertex and let $v$ be a vertex of $G$
with maximum degree, and let  $u\in\,V(G)\backslash\{v\}$ and $w\in\,N(G,u)$.
 As $G$ is in $\mathbb{ES}$ and $d_G(w)\leq\,d_G(v)$, $N(G,w)\subseteq\,N(G,v)\cup\{v\}$. So, $uv\in\,E(G)$, and
 therefore $v$ is universal vertex in $G$. It is easy to check that if $u$ is isolated or universal, then $G-\{u\}$ is in  $\mathbb{ES}$.
 \end{proof} }

{The following characterization of $\mathbb{ES}$ gives some information on the structure of the graphs with
maximum Sombor index (resp. reduced Sombor index) among connected $n$-vertex $\nu$-cyclic graphs.}

{\begin{lem}\label{esll3} The following statements hold:\\
1) $K_1$ is in $\mathbb{ES}$,\\
2) if $G$ is in $\mathbb{ES}$, then $G\cup\,K_1$ is in $\mathbb{ES}$,\\
3) if $G$ is in $\mathbb{ES}$, then $G + K_1$ is in $\mathbb{ES}$,\\
4) each graph of $\mathbb{ES}$ is obtained from $K_1$ by applying rules 2) and 3) a finite number of times.
\end{lem}
\begin{proof}
If $G$ is in $\mathbb{ES}$, then it is easy to see that $G\cup\,K_1$ and $G + K_1$ is in $\mathbb{ES}$. Also, by induction on order of graph and using Lemma \ref{esl3}, one can  easily check that any graph in $\mathbb{ES}$ is obtained from $K_1$ by applying rules 2) and 3) a finite number of times.
\end{proof}}

\begin{thm}\label{nth1}
Let $G\in\,\mathbb{G}(n,\nu)$. If $G$ has maximum Sombor index or reduced Sombor index then $G\in\,{\Omega_4^{n,\nu}}$
or $G\cong\,H_{n,\nu}$.
\end{thm}

\begin{proof}
The proof follows from Lemmas \ref{2lm2}, \ref{8.l4}, \ref{nlm2} and Theorems  \ref{thmain1}, \ref{ath1}.
\end{proof}

In the following theorem, it is proved that Conjecture \ref{con} for the case that $\nu=6$ is correct.

\begin{thm}\label{nth2}
Let $n$ is a positive integer greater than or equal to $8$. Then the graph $H_{n,6}$ has maximum $SO$ and $SO_{red}$ among all $n-$vertex connected graphs with cyclomatic number $6$.
\end{thm}

{\begin{proof}
Suppose $G_1$, $G_2$ and $G_3$ are graphs which are depicted in Figure \ref{fig1}. Also,
suppose  $F_1=(G_1\cup(n-5)K_1)+K_1$, $F_2=(G_2\cup(n-6)K_1)+K_1$, and $F_3=(G_3\cup(n-7)K_1)+K_1$.
By Lemma \ref{esll3}, $\Omega_4^{n,6}=\{F_1,F_2,F_3\}$. Now, by the structures of the graphs $F_1$, $F_2$ and $F_3$, it is easy to see that
{\small\begingroup\allowdisplaybreaks
\begin{align*}
&~~~SO(H_{n,6})-SO(F_1)\\
&=\sqrt { \left( n-1 \right) ^{2}+49}-4\,\sqrt { \left( n-1 \right) ^{2}
+16}+6\,\sqrt {53}-24\,\sqrt {2}-3\,\sqrt { \left( n-1 \right) ^{2}+1}+6\,\sqrt { \left( n-1 \right) ^{2}+4}\\[2mm]
&>7\,\sqrt { \left( n-1 \right) ^{2}+4}-7\,\sqrt { \left( n-1 \right) ^{
2}+16}+6\,\sqrt {53}-24\,\sqrt {2}\\[2mm]
&=-\frac{84}{\sqrt { \left( n-1 \right) ^{2}+4}+\sqrt { \left( n-1 \right) ^{
2}+16}}+6\,\sqrt {53}-24\,\sqrt {2}>4,\\[2mm]
&~~~SO_{red}(H_{n,6})-SO_{red}(F_1)\\
&=6\,\sqrt {37}+6\,\sqrt { \left( n-2 \right) ^{2}+1}+\sqrt { \left( n-2 \right) ^{2}+36} -4\,\sqrt { \left( n-2 \right) ^{2}+9}-18\,\sqrt {2}-3\,\sqrt { \left( n-2 \right) ^{2}}\\[2mm]
&>6\,\sqrt {37}+7\,\sqrt { \left( n-2 \right) ^{2}+1}-18\,\sqrt {2}-7\,\sqrt { \left( n-2 \right) ^{2}+9}>6,\\[2mm]
&SO(H_{n,6})-SO(F_2)>5\sqrt { \left( n-1 \right) ^{2}+4}-5\sqrt { \left( n-1 \right) ^{2
}+16}+6\sqrt {53}-10-\sqrt {29}-2\sqrt {34}-\sqrt {41}
>6,\\[3mm]
&SO_{red}(H_{n,6})-SO_{red}(F_2)>5\,\sqrt { \left( n-2 \right) ^{2}+1}-5\,\sqrt { \left( n-2 \right) ^{
2}+9}-\sqrt {17}-2\,\sqrt {13}+6\,\sqrt {37}-4\,\sqrt {5}-5>8,\\[3mm]
&SO(H_{n,6})-SO(F_7)>3\,\sqrt {\left( n-1 \right) ^{2}+4}-3\,\sqrt {\left( n-1 \right) ^{2}+9}-6\,\sqrt {5}+6\,\sqrt {53}-3\,\sqrt {2}-6\,\sqrt {2}\sqrt {5}>6,\\[3mm]
&SO_{red}(H_{n,6})-SO_{red}(F_7)>3\,\sqrt { \left( n-2 \right) ^{2}+1}-3\,\sqrt { \left( n-2 \right) ^{2}+4}-2\,\sqrt {29}-3\,
\sqrt {26}-{2\,\sqrt {2}+6\,\sqrt {37}}>6.
\end{align*}\endgroup}
Therefore,  $SO(F_i)<SO(H_{n,6})$ and $SO_{red}(F_i)<SO_{red}(H_{n,6})$ for $i\in\{1,2,3\}$. So, the proof follows from Theorem \ref{nth1}.
\end{proof}}

A similar argument as Theorem \ref{nth2} shows that the graph $H_{n,\nu}$, $0 \leq \nu \leq 5$, has the maximum reduced Sombor index among all connected $n-$vertex {graphs}, $n \geq\,\nu-2$. By Theorem \ref{nth1},  Conjecture \ref{con} can be reduced to the following conjecture:

\begin{con}
Suppose $n$ and $\nu$ are non-negative integers such that $7\leq\nu\leq\,n-2$. If $G\in\,{\Omega_4^{n,\nu}}$, then
$SO(G)<SO(H_{n,\nu})$ and $SO_{red}(G)<SO_{red}(H_{n,\nu})$.
\end{con}

 \begin{figure}[h] \centering
 \definecolor{ttqqtt}{rgb}{0.2,0,0.2}
\definecolor{sqsqsq}{rgb}{0.13,0.13,0.13}
\definecolor{ffqqqq}{rgb}{1,0,0}
\scalebox{1.5}{ \begin{tikzpicture}[line cap=round,line join=round,>=triangle 45,x=1.0cm,y=1.0cm]
\clip(-7,0.2) rectangle (1,2.2);
\draw [line width=1.6pt] (-6,2)-- (-5,2);
\draw [line width=1.6pt] (-5,2)-- (-5,1);
\draw [line width=1.6pt] (-5,1)-- (-6,1);
\draw [line width=1.6pt] (-6,1)-- (-6,2);
\draw [line width=1.6pt] (-6,2)-- (-5,1);
\draw [line width=1.6pt] (-5,2)-- (-6,1);
\draw [line width=1.6pt] (-4.49,2.01)-- (-4.49,1.01);
\draw [line width=1.6pt] (-4.49,1.01)-- (-3.49,1.01);
\draw [line width=1.6pt] (-3.49,1.01)-- (-3.49,2.01);
\draw [line width=1.6pt] (-3.49,2.01)-- (-4.49,2.01);
\draw [line width=1.6pt] (-3.49,2.01)-- (-4.49,1.01);
\draw [line width=1.6pt] (-3.49,2.01)-- (-2.49,2.01);
\draw [color=ttqqtt](-5.7,1.0) node[anchor=north west] {$G_1$};
\draw [color=ttqqtt](-4.2,1.0) node[anchor=north west] {$G_2$};
\draw [line width=1.6pt] (-2,1)-- (-1,1);
\draw [line width=1.6pt] (-1,1)-- (0,1);
\draw [line width=1.6pt] (-1,1)-- (-2,2);
\draw [line width=1.6pt] (-1,1)-- (-1,2);
\draw [line width=1.6pt] (-1,1)-- (-0.17,1.83);
\draw [line width=1.6pt] (-2,1)-- (-2,2);
\draw [color=ttqqtt](-1.23,1.0) node[anchor=north west] {$G_3$};
\begin{scriptsize}
\fill [color=ffqqqq] (-6,2) circle (2.0pt);
\fill [color=ffqqqq] (-5,2) circle (2.0pt);
\fill [color=ffqqqq] (-5,1) circle (2.0pt);
\fill [color=ffqqqq] (-6,1) circle (2.0pt);
\fill [color=ffqqqq] (-4.49,2.01) circle (2.0pt);
\fill [color=ffqqqq] (-4.49,1.01) circle (2.0pt);
\fill [color=ffqqqq] (-3.49,1.01) circle (2.0pt);
\fill [color=ffqqqq] (-3.49,2.01) circle (2.0pt);
\fill [color=ffqqqq] (-2.49,2.01) circle (2.0pt);
\fill [color=ffqqqq] (-2,1) circle (2.0pt);
\fill [color=ffqqqq] (-1,1) circle (2.0pt);
\fill [color=ffqqqq] (0,1) circle (2.0pt);
\fill [color=ffqqqq] (-2,2) circle (2.0pt);
\fill [color=ffqqqq] (-1,2) circle (2.0pt);
\fill [color=ffqqqq] (-0.17,1.83) circle (2.0pt);
\end{scriptsize}
\end{tikzpicture}}
\centering\caption{ All $n-$vertex  connected graphs in $\mathbb{ES}$, $n \leq 6$, with exactly six edges. \label{fig1}}
 \end{figure}

\section{Concluding Remarks}

In this paper, the authors develop a method to  calculate the maximum value of the Sombor index of graphs. This method is applied to solve partially a conjecture of R\'eti et al. \cite{16}. In an exact phrase, it is proved that  the graph $H_{n,6}$ has the maximum Sombor index among all connected $6$-cyclic graphs of order $n$. {Moreover, we proved that the graph $H_{n,\nu}$ has the maximum Sombor index among all connected $\nu$-cyclic graphs of order $n$, where $5\leq\,\nu\leq n-2$ and the number of pendent vertices is less than or equal to $n - \nu - 2$.} It is also proved that under the same conditions this conjecture is valid for  the reduced  Sombor index.

\section*{Acknowledgments} {The authors are grateful to the referees for their valuable comments and suggestions.}

\section*{Conflicts of Interest}

{The authors declare that there is no conflict of interest.}

\end{document}